\documentclass[twoside,11pt,reqno]{amsart}
\usepackage{xypic}

\usepackage[bookmarks,
bookmarksnumbered,%
 colorlinks=true,%
 linkcolor=red,%
 citecolor=blue,%
 filecolor=blue,%
 urlcolor=blue,%
]
{hyperref}

\usepackage{amssymb, amsmath, latexsym, mathtools}
\usepackage{epigraph}

\makeatletter
\@namedef{subjclassname@2020}{\textup{2020} Mathematics Subject Classification}
\makeatother

\DeclareMathAlphabet{\curly}{OT1}{rsfs}{n}{it}
 
\DeclareMathOperator{\Dim}{dim}

\DeclareMathOperator{\Hom}{Hom}

\DeclareMathOperator{\rank}{rank}

\DeclareMathOperator{\im}{Im}

\DeclareMathOperator{\Fix}{Fix}
\DeclareMathOperator{\tr}{tr}

\DeclareMathOperator{\inc}{in}
\DeclareMathOperator{\pr}{pr}
\DeclareMathOperator{\coker}{Coker}
\DeclareMathOperator{\Sm}{Sm}

\DeclareMathOperator{\Sq}{Sq}
\DeclareMathOperator{\Bl}{Bl}
\DeclareMathOperator{\Tors}{Tors}
\DeclareMathOperator{\interior}{Int}

\DeclareMathOperator{\defi}{\mathfrak D}
\DeclareMathOperator{\Conj}{c}

\DeclareMathOperator{\Ker}{Ker}

\def\dim{\mbox{dim}}
\def\ra{\rightarrow}

\def\CC{\mathbb{C}}
\def\PP{\mathbb{P}}
\def\QQ{\mathbb{Q}}
\def\ZZ{\mathbb{Z}}
\def\NN{\mathbb{N}}  
\def\RR{\mathbb{R}}

\def\FF{\mathbb{F}}

\def\s-{\setminus}

\def\HH{\mathbb{H}}

\def\ra{\rightarrow}

\newtheorem{thm}{Theorem}[section]
\newtheorem{prop}[thm]{Proposition}
\newtheorem{lemma}[thm]{Lemma}
\newtheorem{defn}[thm]{Definition}

\newtheorem{cor}[thm]{Corollary}
\newtheorem{lem}[thm]{Lemma}

\newtheorem{rmk}[thm]{Remark}

\numberwithin{equation}{section}

\newenvironment{notations}{\mbox{ }\\{\bf  Notations and conventions:}\mbox{ }}

\begin{document}

\title[Deficiency]
{On the Smith-Thom deficiency of Hilbert squares}

\author[Kharlamov]{Viatcheslav Kharlamov}

\address{
        IRMA UMR 7501, Strasbourg University, 7 rue Ren\'e-Descartes, 
        67084 Strasbourg, Cedex,  FRANCE}

\email{kharlam@math.unistra.fr}

\author[R\u asdeaconu]{Rare\c s R\u asdeaconu}

\address{        
       Department of Mathematics, Vanderbilt University, 
        Nashville, TN, 37240, USA}
  \address{
	Institute of Mathematics of the Romanian Academy,  Bucharest 014700,  Romania}

\email{rares.rasdeaconu@vanderbilt.edu}

\keywords{real algebraic varieties, Smith exact sequence, Smith-Thom maximality,  
Hilbert scheme of points}

\subjclass[2020]{Primary: 14P25; Secondary: 14C05, 14J99}

\begin{abstract}
We give an expression for the Smith-Thom deficiency of the Hilbert square $X^{[2]}$ 
of a smooth real algebraic variety $X$ in terms of the rank of a suitable Mayer-Vietoris 
mapping in several situations. As a consequence, we establish a simple necessary 
and sufficient condition for the maximality of $X^{[2]}$ in the case of projective complete 
intersections, and show that with a few exceptions no real nonsingular projective 
complete intersection of even dimension has maximal Hilbert square. We also provide new 
examples of smooth real algebraic varieties with maximal Hilbert square.
\end{abstract}

\maketitle

\thispagestyle{empty}

\vskip-4mm
\setlength\epigraphwidth{.56\textwidth}
\epigraph{
\dots si beau soit un vaste paysage, les horizons lointains sont toujours un peu vagues \dots}
{Henri Poincar\'e, Savants et \'ecrivains.}

\setcounter{tocdepth}{2}
\tableofcontents


\section{Introduction}
\label{intro}


Recall that on a real algebraic variety $X$ the following inequality holds
$$
\dim\, H_*(X(\RR); \FF_2)\le \dim\, H_*(X(\CC); \FF_2).
$$
It is traditionally called {\it the Smith inequality} and the difference
$$
\defi(X)= \dim\, H_*(X(\CC); \FF_2)-\dim\, H_*(X(\RR); \FF_2)
$$ 
is called {\it the Smith-Thom deficiency of} $X.$ A real algebraic variety $X$ is said 
to be {\it maximal}, or an {\it $M$-variety}, if its Smith-Thom deficiency vanishes.

\smallskip

As it was observed in our previous paper \cite{loss-surfaces}, many  deformation classes 
of algebraic surfaces do not contain any real representative $X$ whose Hilbert square 
$X^{[2]}$ is maximal. Here, we refine results obtained in \cite{loss-surfaces} and generalize 
them to higher dimensions. The main results are as follows.

\begin{thm}
\label{converse}
Let $X$ be a real nonsingular projective variety of dimension $n\ge 2$.
If the Hilbert square $X^{[2]}$ is maximal, then $X$ is maximal.
\end{thm}

\begin{thm}
\label{defect-odd}
The Smith-Thom deficiency of the Hilbert square $X^{[2]}$  of a maximal real nonsingular 
projective complete intersection  $X\subset \PP^{N}$ of dimension $n\ge 2$ is given by
$$
\defi(X^{[2]})=4\left(\sum_{l=1}^{[\frac{n}{2}]}\sum_{i=0}^{l-1}\beta_i(X(\RR))-
d(n)\right),
$$
where $d(n)$ is equal to $\frac{n(n+2)}{8}$ for $n$ even and $\frac{n^2-1}{8}$ 
for $n$ odd.
\end{thm}

The $\FF_2$-Betti numbers $\beta_i(X(\RR))$ are non-zero for every maximal real 
nonsingular complete intersection $X\subset \PP^N$ of dimension 
$n$ whenever $0\leq i\leq [\frac{n}{2}]$  (see Lemma \ref{hyperplane-class}). Since 
$d(n)=\sum_{l=1}^{[\frac{n}{2}]} l$, we get the following criterium of maximality as 
a consequence of Theorems \ref{converse} and \ref{defect-odd}.

\begin{cor}
\label{odd-coro}
Let $X\subset \PP^{N}$ be a real nonsingular projective complete intersection of 
dimension $n\ge 2.$ Then $X^{[2]}$ is maximal if and only if $X$ is maximal and 
$$
\beta_i(X(\RR))=1\quad\text{for every }\quad 0\leq i\leq \left[\frac{n}{2}\right]-1.
$$
\end{cor}
In particular, Corollary \ref{odd-coro} yields a  direct proof of the following result due to 
L. Fu \cite[Theorem 7.5]{fu}:
\begin{cor}
\label{complete-intersections}
The Hilbert square of a real cubic threefold is maximal if and only if the threefold 
is maximal.
\end{cor}

According to Theorem 11.12 in \cite{fu}, no real cubic four-fold has maximal 
Hilbert square. As a consequence of Theorems \ref{converse} and  \ref{defect-odd}, 
we prove the following:

\begin{thm}
\label{highdegree2kfolds}
The Hilbert square $X^{[2]}$ of a real nonsingular projective complete 
intersection $X\subset \PP^N$ of positive even dimension is maximal if and only 
if $X$ is a linear subspace, a maximal quadric, a maximal intersection of two quadrics, 
or a maximal cubic surface in $\PP^3.$
\end{thm}

As it was communicated to us by L. Fu, for a cubic hypersurface $X,$ the 
Galkin-Shinder-Voisin diagram (see \cite[section 4]{bfr}) yields the following 
relation between the Smith-Thom deficiencies of $X,$ its 
Hilbert square, and its Fano variety of lines $F(X):$
$$
\defi(X^{[2]})=(n+1)\defi(X)+\defi(F(X)). 
$$
Therefore, Theorem \ref{highdegree2kfolds} has the following consequence:

\begin{cor}
\label{fano-var-cubics}
Let $X$ be a real nonsingular cubic hypersurface of dimension $2n$ with $n\geq 2.$ 
If X is maximal \footnote{By common belief, maximal real cubic hypersurfaces should 
exist in all dimensions. However, up to our knowledge, published proofs cover the 
statement only up to dimension 4.}, then its Fano variety of lines  $F(X)$ is not maximal.
 \end{cor}

We study next the maximality of the Hilbert square of real algebraic varieties whose 
complex locus has no odd degree cohomology.

\begin{thm}
\label{connected}
Let $X$ be a real nonsingular projective variety of dimension $\ge 2$
\footnote{The only exception in dimensions $<2$ is that of 
$X=\PP^1$ with a real structure without real points.}
satisfying $H_{\rm odd}(X)=0.$ If the Hilbert square $X^{[2]}$ is maximal, then 
$\beta_k(X(\RR))=\beta_{2k}(X(\CC))$ for every $k\geq 0.$ In particular, 
$X(\RR)$ is connected.
\end{thm}

The conditions found in Theorem \ref{connected} are not sufficient to ensure the 
maximality of the Hilbert square in general ({\it cf.}, Theorem \ref{highdegree2kfolds}). 
However, we identify a case when they are sufficient.

\begin{thm}
\label{Pn}
Let $X$ be a maximal real nonsingular projective variety such that $H_{\rm odd}(X)=0.$ 
If for every $k\le \dim \, X$, the group $H_{k}(X(\RR))$ is generated by fundamental 
classes of real loci of real smooth algebraic submanifolds, then the Hilbert square $X^{[2]}$ 
is maximal.
\end{thm}

\begin{cor}
The Hilbert square is maximal for any non-singular toric variety equipped with the 
standard real structure, maximal non-singular quadric hypersurfaces in $\PP^{2n+1},$ 
and finite products of such varieties equipped with the product real structure.
\qed
\end{cor}

More examples of real projective manifolds with maximal Hilbert square will 
be presented elsewhere \cite{effectivity}, in a different context.

\smallskip

The results obtained in Theorems \ref{converse}, \ref{connected}, and \ref{Pn}, including 
their proofs, literally extend from the real algebraic setting to compact complex manifolds 
equipped with an anti-holomorphic involution.


\subsection*{Acknowledgments} 
We are greatly thankful to L. Fu for suggestions and comments on this work. The first 
author acknowledges the support from the grant ANR-18-CE40-0009 of French 
Agence Nationale de Recherche, while the second author acknowledges the support 
of a Professional Travel Grant from Vanderbilt University. A part of this work was 
completed during the author's participation at the ``Research in Paris" program at the 
Institut Henri Poincar\'e, and we thank this institution for the hospitality and support.


\notations


\begin{itemize}
\item[ 1)] By a complex variety equipped with a real structure, we mean a pair 
$(Y,\Conj)$ consisting of a complex variety $Y$ and an anti-holomorphic involution 
$\Conj:Y\ra Y.$ When the anti-holomorphic conjugation is understood from the 
context, we will simply say that $Y$ is defined over the reals.
\item[ 2)]  Let $Y$ be an algebraic variety defined over $\RR,$ and $G$ denote 
the Galois group ${\text{Gal}}(\CC/\RR).$ The group $G$ is a cyclic group of 
order $2$ and  acts on the locus of complex points $Y(\CC).$ The non-trivial 
element of $G$ acts as an anti-holomorphic involution, which we will denote 
by $c,$ and the fixed point set of the action coincides with the set of real points 
of $Y.$ The pair $(Y, \Conj)$ is a variety equipped with a real structure. To mediate 
between the notations traditionally used for varieties equipped with real structures 
and for algebraic varieties defined over $\RR,$ we will use from now on $Y$ to 
denote the set of complex points, and $Y(\RR)$ the set of real points. 
\item[ 3)] Unless  explicitly stated, all the homology and cohomology groups have 
coefficients in the field $\FF_2=\ZZ/2\ZZ$. We use  $\beta_i(\,\cdot\,)$ and 
$b_i(\,\cdot\,)$  to denote the Betti numbers when the coefficients are in 
$\FF_2$ or in $\QQ,$ respectively. For convenience, we allow the index $i$ 
to be an arbitrary integer, by setting $\beta_i=0$ for $i< 0.$ We will use  the 
notations $\beta_*(\,\cdot\,)$  and $b_*(\,\cdot\,)$ for the corresponding total Betti 
numbers, while $\beta_{\rm odd}(\,\cdot\,)$ and $\beta_{\rm even}(\,\cdot\,)$ 
denote the sums $\displaystyle \sum_{i\geq 0} \beta_{2i+1}(\,\cdot\,)$ and 
$\displaystyle \sum_{i\geq 0} \beta_{2i}(\,\cdot\,),$ respectively.
\end{itemize}


\section{Preliminaries}


\subsection{Smith theory}
\label{Smith.theory}\label{s7}


Most results cited in this section are due to P.A. Smith. Proofs can be found, for example, in 
\cite[Chapter 3]{bredon} and \cite[Chapter 1]{dik}.

\smallskip

Throughout the section we consider a topological space $X$ with a cellular involution 
$c:X\to X,$ that is, $X$ is a finite CW-complex, $\Conj$ transforms cells into cells and acts 
identically on each invariant cell \footnote{This rather traditional condition that $X$ is 
a CW-complex (or a simplicial complex, as in \cite{bredon}) can be relaxed at the cost 
of using \v Cech cohomology and assuming $X$ to be finite dimensional locally compact 
Hausdorff topological space. In this paper, Smith theory is applied to smooth manifolds 
and smooth involutions, so the CW-complex assumption is largely enough.}. Let  
$F=\Fix c,\,\bar X=X/c,$ and denote by $\inc: F\hookrightarrow X$ and $\pr: X\ra \bar X$ 
the natural inclusion and projection, respectively.

\smallskip

Consider the  \emph{Smith chain complexes} defined by 
\begin{align*}
\Sm_*(X)&=\Ker[(1+c_*)\, :S_*(X)\to S_*(X)],\\
\Sm_*(X,F)&=\Ker[(1+c_*)\, :S_*(X,F)\to S_*(X,F)].
\end{align*}
and their \emph{Smith homology} $H_r(\Sm_*(X))$ and $H_r(\Sm_*(X,F)),$ respectively. 
There exists a canonical isomorphism $\Sm_*(X,F)=\im[(1+c_*)\, :S_*(X)\to S_*(X)].$ 
The \emph{Smith sequences} are the long homology and cohomology exact sequences 
associated with the short exact sequence of complexes
\begin{equation}
\label{smith-sequence}
0\ra\Sm_*(X)\xrightarrow{\text{inclusion}}S_*(X)
\xrightarrow{1+c_*}\Sm_*(X,F)\ra 0.
\end{equation}

There is also a canonical splitting, $\Sm_*(X)=S_*(F)\oplus\im(1+c_*),$ and the 
transfer homomorphism $\tr^*:S_*(\bar X,F)\to\Sm_*(X,F)$ is an isomorphism 
\cite[Chapter 3]{bredon} (see also {\it op. cit.} for the cohomology version). 
In view of these identifications the long exact sequences associated to 
(\ref{smith-sequence}) yield:

\begin{thm}
\label{Smith.seq}
There are two natural, in respect to equivariant maps, exact sequences, called  
\emph{(homology and cohomology) Smith sequences of $(X,c)$:}
$$
\begin{gathered}
\cdots \ra H_{p+1}(\bar X,F)\xrightarrow[]{\Delta} H_p(\bar X,F)\oplus H_p(F)
  \xrightarrow{\tr^*+\inc_*} H_p(X)
 \xrightarrow{\pr_*} H_p(\bar X,F)\ra \rlap{\,},\\
\ra H^p(\bar X,F)\xrightarrow{\pr^*}H^p(X)\xrightarrow{{\tr_*}\oplus{\inc^*}} 
  H^p(\bar X,F)\oplus H^p(F)\xrightarrow{\Delta} H^{p+1}(\bar X,F)\ra\cdots  \rlap{\,.}
\end{gathered}
$$

The homology and cohomology connecting homomorphisms $\Delta$ are given by
$$
x\mapsto x\cap\omega\oplus\partial x\quad\text{and}\quad x\oplus f\mapsto
x\cup\omega+\delta f,
$$
respectively, where $\omega \in H^1(\bar X\setminus F)$ is the characteristic class of 
the double covering $X\setminus F\to\bar X\setminus F$. The images of ${\tr^*}+\inc_*$ 
and~$\pr^*$ consist of invariant classes:
$\im\tr^*\subset\Ker(1+c_*)$ and $\im\pr^*\subset\Ker(1+c^*)$.
\end{thm}

The following immediate consequences of  Theorem \ref{Smith.seq}, which we 
state in the homology setting, have an obvious counterpart for cohomology.

\begin{cor}
\label{cor-smith}
Let $(X, c)$ be a topological space equipped with a cellular involution. Then
\begin{equation}
\label{Smith-dims}
\Dim~H_*(F)+2\sum_{p}\Dim\coker({\tr^p}+\inc_p)=\Dim~H_*(X).
\end {equation}
As a consequence, we have 
\begin{equation}
\label{Smith-ineq}
\Dim~H_*(F)\le\Dim~H_*(X)\quad{\text{\rm (Smith inequality)}}.
\end{equation}
\end{cor}

\begin{defn}
Let $(X,\Conj)$ be a topological space equipped with a cellular involution. The integer 
$$
\defi(X,\Conj)=2\sum_{p}\Dim\coker({\tr^p}+\inc_p)
$$
is called the Smith-Thom deficiency of  $(X,\Conj).$
If $\defi(X,\Conj)=0$, the topological space 
$X$ is called \emph{maximal}, or an \emph{$M$-space}, and  $\Conj$ is called an 
\emph{$M$-involution}.
\end{defn}
When the involution is understood from the context, it will be omitted from the notation 
of the Smith-Thom deficiency.

\smallskip 

Notice from Corollary \ref{cor-smith} that $X$ is maximal if and only if 
$\Dim~H_*(F)=\Dim~H_*(X),$ and from Theorem \ref{Smith.seq} we find the 
following characterization of maximality.

\begin{cor}
\label{maxSmith}
Let $(X,c)$ be a topological space equipped with a cellular involution. Then 
$X$ is maximal if and only if the sequence 
$$
0\ra H_{k+1}(\bar X, F)\xrightarrow{\Delta} H_{k}(\bar X, F)\oplus H_k(F)\ra H_k(X)\ra 0
$$
is exact
for every $k\geq 0$.
\end{cor} 

\begin{lem} 
\label{RelativeQuotient}
If a $d$-dimensional space $(X, c)$ is maximal and $r\leq d,$ then
\begin{equation}
\label{relative-aux}
\beta_r(\bar X, F)=\sum_{k=r}^d (\beta_k(X)- \beta_k(F)).
\end{equation}
If, in addition, $d=2n$, $X$ is a smooth closed manifold, $c$ a smooth involution, 
and each component of $F$ is $n$-dimensional, then we have
\begin{equation}
\label{relative-individual}
\beta_{r}(\bar X, F)=\sum_{k=r}^{2n}\beta_k(X), \quad\text{for every $r\ge n+1$,}
\end{equation}
and 
\begin{equation}
\label{relative-main}
\beta_*(\bar X, F)=\frac{n}2\beta_*(X).
\end{equation}
\end{lem}

\begin{proof} According to Corollary \ref{maxSmith}, under maximality 
assumption, for every $k\geq 0$ we have $\beta_{k+1}(\bar X, F) +\beta_k(X)
=\beta_{k}(\bar X, F) +\beta_k(F).$ We obtain (\ref{relative-aux}) by summing 
up these equalities over all $k$ with $r\le k\le d$ and noticing that 
$\beta_{d+1}(\bar X, F) =0.$ In the manifold case, (\ref{relative-individual}) 
is a direct consequence of  (\ref{relative-aux}).

To prove (\ref{relative-main}), by adding the equalities (\ref{relative-aux}) 
over all $r$ with $1\le r\le d,$ we obtain
\begin{equation}
\label{relative-sum}
\beta_*(\bar X, F)=\sum_{r=0}^{2n} r(\beta_r(X)- \beta_r(F)).
\end{equation}
Due to Poincar\'e duality applied to both $X$ and $F$, we have$$
\sum_{r=0}^{2n}r \beta_r(X)=\sum_{r=0}^{2n}(2n-r)\beta_r(X)=2n\beta_*(X) 
-\sum_{r=0}^{2n}r \beta_r(X)
$$
and
$$
\sum_{r=0}^{2n}r \beta_r(F)=\sum_{r=0}^{n}r \beta_r(F)=
\sum_{r=0}^{n}(n-r)\beta_r(F)=n\beta_*(F)-
\sum_{r=0}^{n}r \beta_r(F),
$$
respectively. Hence, 
\begin{equation}
\label{sums}
\sum_{r=0}^{2n}r \beta_r(X)=n\beta_*(X), \quad \sum_{r=0}^{2n}r \beta_r(F)
=\frac{n}2\beta_*(F),
\end{equation}
and  (\ref{relative-main}) follows from
(\ref{relative-sum}), (\ref{sums}),  and the maximality assumption.
\end{proof}


\subsection{A non-vanishing result}


An ingredient needed for the proofs of results announced in Section \ref{intro}
is an observation of the first author \cite{Kh75} regarding the non-vanishing 
of the cohomology of real hypersurface. We include a short proof here for 
the convenience of the reader.

\begin{lemma}
\label{hyperplane-class} 
Let $X\subset \PP^N$ be a maximal complete intersection of dimension 
$n$ and $\upsilon\in H^1(X(\RR))$  the class of a hyperplane section
$H(\RR)\subset X(\RR)$. Then $\upsilon^r\neq 0$  for every 
$r \in \NN, r \leq [\frac{n}{2}].$
In particular, $\beta_r(X(\RR))\ge 1$ for every $ r \in \NN, 
r \leq [\frac{n}{2}]$.
\end{lemma}
\begin{proof}[Proof \rm ({\it cf., \cite{Kh75}})]
Let $d$ be the degree of $X\subset \PP^N.$  If $d$ is odd, then 
$\upsilon^r\neq 0$ for every 
$0\le r\le n$, just because $(\upsilon^r\cup \upsilon^{n-r})\cap 
[X(\RR)]=d\mod 2.$\\
\indent If $d$ is even, the result is a straightforward consequence of the 
following properties:
\begin{itemize}
\item If $k\neq n$ is odd, then $H^k(X;\ZZ)$=0.
\item If $k$ is even and $n< k\le 2n$, then $H^k(X;\ZZ)$ is isomorphic 
to $\ZZ$ and generated by $h^{\frac{k}2}$
where $h\in H^2(X;\ZZ)$ is the class of hyperplane sections $H\subset X$.
\item If $k$ is even and $0\le k<n,$ then $H^k(X;\ZZ)$ is isomorphic to $\ZZ$ and 
$h^{\frac{k}2}$ is a $d$-multiple of a generator.
\item if $n$ is even, then $h^{\frac{n}2}$ is a primitive element of $H^n(X;\ZZ)$ 
(see, for example, \cite[Lemma 1.2]{Kh75}).
\end{itemize} 

Indeed, from the enumerated above properties and  exactness of the 
$\FF_2$-cohomology sequence of the pair $(\PP^N, X)$ it follows that 
\begin{equation*}
\beta_*(\PP^N, X) = \beta_*(\PP^N)+\beta_*(X) - 2 \left(\left[\frac{n}2\right] +1\right).
\end{equation*}
On the other hand, from the exactness of the $\FF_2$-cohomology sequence 
of the pair $(\PP^N(\RR), X(\RR))$ we get
\begin{equation*}
\beta_*(\PP^N(\RR), X(\RR)) = \beta_*(\PP^N(\RR))+\beta_*(X(\RR)) - 2 (\ell+1),
\end{equation*}
where $\ell$ is determined by the property that $\upsilon^r\ne 0$ for $r\le l$ 
and  $\upsilon^r= 0$ for $r>\ell.$ Now, it remains to apply the Smith inequality
in the relative setting (see \cite{bredon}, Chap. 3, Theorem 4.1),
\begin{equation*}
\beta_*(\PP^N(\RR), X(\RR)) \le \beta_*(\PP^N, X) 
\end{equation*}
and to use the maximality assumption, $\beta_*(X(\RR))= \beta_*(X)$.
\end{proof}


\subsection{Elementary computations}


The following computations are used in the proof of Theorem \ref{defect-odd}.

\begin{lemma}
\label{3together}
Let $X$ be a maximal real nonsingular projective variety of odd dimension $n$. 
Then, the following relations hold:
\begin{align}
\sum_{l=0}^{\frac{n-1}{2}}\sum_{i+j=2l}\beta_i(X(\RR))\beta_j(X(\RR))=&\frac14 \beta^2_*\, ,\label{A} \\
\sum_{l=0}^{\frac{n-1}{2}}\beta_{l}(X(\RR))=&\frac12\beta_*\, , \label{B}\\
\sum_{l=0}^{\frac{n-1}{2}}\sum_{i=0}^{2l-1}\beta_{i}(X(\RR))=&
\frac{n-1}{4}\beta_*.
\end{align}
\end{lemma}

\begin{proof} By Poincar\'e duality, we find:
\begin{align*}
2\sum_{l=0}^{\frac{n-1}{2}}&\,\sum_{i+j=2l}\beta_i(X(\RR))\beta_j(X(\RR))\\ 
= &\, \sum_{l=0}^{\frac{n-1}{2}}\sum_{i+j=2l}\beta_i(X(\RR))\beta_j(X(\RR))
+\sum_{l=0}^{\frac{n-1}{2}}\sum_{i+j=2l}\beta_{n-i}(X(\RR))\beta_{n-j}(X(\RR))\\
=&\, \sum_{l=0}^{\frac{n-1}{2}}\sum_{i+j=2l}\beta_i(X(\RR))\beta_j(X(\RR)) 
+\sum_{l'=\frac{n+1}{2}}^{n}\sum_{i'+j'=2l'}\beta_{i'}(X(\RR))\beta_{j'}(X(\RR))\\
=&\, \sum_{k=0}^n\sum_{a+b=2k}\beta_a(X(\RR))\beta_b(X(\RR))\\
=&\,\beta^2_{\rm even}(X(\RR))+\beta^2_{\rm odd}(X(\RR)).
\end{align*}
Since $X(\RR)$ is an odd-dimensional manifold, using again the Poincar\'e duality 
we see that $\beta_{\rm odd}(X(\RR))=\beta_{\rm even}(X(\RR)),$ while 
by the maximality of $X$ we have $\beta_{\rm odd}(X(\RR))+\beta_{\rm even}
(X(\RR))=\beta_*.$ 
Therefore 
$$
\beta_{\rm odd}(X(\RR))=\beta_{\rm even}(X(\RR))=\frac12 \beta_*,
$$ 
wherefrom (\ref{A})  follows immediately. 

\smallskip

Again by Poincar\'e duality, we have 
\begin{align*}
2\sum_{l=0}^{\frac{n-1}{2}}\beta_{l}(X(\RR))=&\, \sum_{l=0}^{\frac{n-1}{2}}\beta_{l}(X(\RR))
+\sum_{l=0}^{\frac{n-1}{2}}\beta_{n-l}(X(\RR))\\
=&\, \sum_{l=0}^{\frac{n-1}{2}}\beta_{l}(X(\RR))+\sum_{l'=\frac{n+1}{2}}^{n}\beta_{l'}(X(\RR))\\
=&\, \beta_*(X(\RR)).
\end{align*}
This implies (\ref{B}) due to the maximality of $X$.

\smallskip

A direct computation shows 
\begin{equation}
\label{reduction-E}
\sum_{l=0}^{\frac{n-1}{2}}\sum_{i=0}^{2l-1}\beta_{i}(X(\RR))=
\sum_{j=0}^{n}\left\lceil\frac{n-1-j}{2}\right\rceil\beta_{j}(X(\RR)).
\end{equation}
Arguing as before, by Poincar\'e duality, we obtain
\begin{align}
\label{PD-E}
2\sum_{j=0}^{n}&\, \left\lceil\frac{n-1-j}{2}\right\rceil\beta_{j}(X(\RR))\notag \\
=&\, \sum_{j=0}^{n}\left\lceil\frac{n-1-j}{2}\right\rceil\beta_{j}(X(\RR))+
\sum_{j=0}^{n}\left\lceil\frac{n-1-j}{2}\right\rceil\beta_{n-j}(X(\RR))\notag \\
=&\, \sum_{j=0}^{n}\left\lceil\frac{n-1-j}{2}\right\rceil\beta_{j}(X(\RR))+
\sum_{i=0}^{n}\left\lceil\frac{i-1}{2}\right\rceil\beta_{i}(X(\RR))\notag \\
=&\, \sum_{a=0}^{n}\left(\left\lceil\frac{n-1-a}{2}\right\rceil+
\left\lceil\frac{a-1}{2}\right\rceil\right)\beta_a(X(\RR))\notag\\
=&\, \frac{n-1}{2} \sum_{a=0}^{n}\beta_a(X(\RR)).
\end{align}
The conclusion of the lemma follows now from (\ref{reduction-E}), (\ref{PD-E}) 
and the maximality of $X.$
\end{proof}

When $X$ is even dimensional, we need similar identities. We will only state them, 
as the proof is elementary and follows the one in the odd dimensional case.

\begin{lemma}
\label{3together-even}
Let $X$ be a maximal real nonsingular projective variety of even dimension $n$. 
Then, the following relations hold:
\begin{align}
\sum_{l=1}^{\frac{n}{2}}\sum_{\substack{a+b=2l-1\\a<b}}\beta_a(X(\RR))\beta_b(X(\RR))=&\,
\frac{1}{2}\beta_{\rm even}(X(\RR))\beta_{\rm odd}(X(\RR)) \,,\label{A-even} \\
\sum_{l=1}^{\frac{n}{2}}\sum_{i=0}^{2l-2}\beta_{i}(X(\RR))=&\,
\frac{n}{4}\beta_*-\frac12\beta_{\rm odd}(X(\RR)). \qed
\end{align}
\end{lemma}


\section{Cut-and-Paste construction of Hilbert squares over the reals}
\label{cut-paste}


For smooth varieties a simple, well known, construction of the Hilbert square 
consists in the following. Given a smooth variety $X$, one lifts the involution 
$\tau : X\times X\to X\times X$ permuting the factors to an involution $Bl(\tau)$ 
on the blowup $\Bl_\Delta (X\times X)$ of  $X\times X$ along the diagonal 
$\Delta\subset X\times X.$ The quotient of $\Bl_\Delta (X\times X)$ by $\Bl(\tau)$  
is then naturally isomorphic to the Hilbert square $X^{[2]}.$ By construction, 
the  branch locus $E\subset X^{[2]}$ of the double ramified covering 
$\Bl_\Delta (X\times X)\to X^{[2]}$ coincides with the exceptional divisor of the 
blowup, and, since the normal vector bundle of  $\Delta$ in 
$X\times X$ is isomorphic to the tangent vector bundle $TX$ of $X$, both this 
exceptional divisor and $E$ are naturally isomorphic to $\PP(T^*X).$\footnote{In 
agreement with \cite{square}, which we use as reference, we follow the Grothendieck 
convention according to which the projectivization of a vector space is the space 
of its hyperplanes.} In other words, a point of $E$ is identified with a point of $X$ plus 
a complex line in the tangent space at that point. Unordered pairs of distinct points in $X$ represent 
the points of the complement, $X^{[2]}\setminus E.$ Notice also that the normal 
bundle of $E$ in $X^{[2]}$ is naturally isomorphic to the square of the tautological 
bundle over $\PP(T^*X)$.

\smallskip

This construction is independent on the choice of the ground field. Applying it to a 
smooth variety $X$ defined over the reals, it equips $X^{[2]}$ with a naturally induced real 
structure, still denoted by $\Conj,$ that acts on points in $X^{[2]}$ represented by pairs of 
distinct points in $X$ by sending a pair to the complex conjugate one, while the points represented 
by a point of $X$ with a tangent line are sent to the conjugate point with the conjugate 
line. Thus, $X^{[2]}(\RR)=X/\Conj$ if $X(\RR)$ is empty, and otherwise $X(\RR)$ 
has the following properties:
\begin{itemize}
\item $E(\RR)\subset X^{[2]}(\RR)$ is naturally diffeomorphic to $\PP_\RR(T^*X(\RR))$ 
and has codimension 1.
\item The complement, $X^{[2]}(\RR)\setminus E(\RR)$, is a disjoint union 
of $ (X/\Conj) \setminus X(\RR)$, $F_i^{(2)}\setminus \Delta F_i$ ($1\le i\le r$), 
and $F_i\times F_j$ ($1\leq i<j\leq r$) where $F_1,\dots, F_r$ are the connected 
components of $X(\RR)$.
\item Each of the connected components $\PP_\RR(T^*F_i)$ of $E(\RR)$ is a 
common boundary of $ (X/\Conj) \setminus X(\RR)$ with  $F_i^{(2)}\setminus \Delta F_i$.
\end{itemize}

From this we can conclude that $X^{[2]}(\RR)$  is a disjoint union of connected components
$$
 X^{[2]}(\RR)=X^{[2]}_{\text{main}}(\RR) \bigsqcup X^{[2]}_{\text{extra}}(\RR),\quad  X^{[2]}_{\text{extra}}(\RR)
 = \bigsqcup_{1\leq i<j\leq r} \left(F_i\times F_j\right)
$$
where $X^{[2]}_{\text{main}}(\RR)$  is the component of  $X^{[2]}(\RR)$ that contains 
$E(\RR)$ in such a way that $E(\RR)$ divides $X^{[2]}_{\text{main}}(\RR)$ in $r+1$ 
submanifolds with boundary: 
$$
X^{[2]}_{\text{main}}(\RR)= \bigcup_{i=0}^{r} \HH_i, 
$$
where 
\begin{align*}
\partial \HH_{0} = E(\RR),& \quad  \interior \HH_{0}\cong (X/\Conj) \setminus X(\RR) \\
\partial \HH_{i}=\PP_\RR(T^*F_i),& \quad  \interior \HH_{i}\cong F_i^{(2)}\setminus \Delta F_i, i=1,\dots, r.
\end{align*}
Here $\Delta F_i$ is the diagonal in $F_i^{(2)},$ for every $i=1,\dots, r.$ Each manifold
$\HH_i,\,  i=1,\dots,r$ is glued to $\HH_0$ along their common boundary
$\PP_\RR( T^*F_i )\subset \PP_\RR(T^*X(\RR)).$
For convenience, let $\displaystyle \HH=\sqcup_{i=1}^r \HH_i,$ and notice that $\partial\HH=E(\RR).$

\smallskip

Denote by
$$
\inc_0:E(\RR)=\partial \HH_0\ra \HH_0,
\quad 
\inc: E(\RR)=\partial (\sqcup_{i=1}^r \HH_i)\ra \bigsqcup_{i=1}^r \HH_i
$$ 
the inclusion maps. Let $\inc_{0}^k$ and $\inc^k$ be the induced maps between the 
$k^{th}$ cohomology groups, and $\inc_0^*$ and $\inc^*$ denote the induced maps 
$$
\inc_{0}^*:  H^*(\HH_0)\ra H^*(E(\RR))\quad{\text{and}}\quad  \inc^{*}: H^*(\HH)\ra H^*(E(\RR)),
$$ 
respectively, where
$$
H^*(E(\RR))=\bigoplus_{i\geq0}
H^i(E(\RR)) \quad{\text{and}} \quad H^*(\HH)= \bigoplus_{i\geq0}
H^i(\HH).
$$
We will write the corresponding induced maps in homology by lowering the degree index. 
We let $\displaystyle \mu =(\inc_{0},\inc) : E(\RR)\ra \HH_0\sqcup(\sqcup_{i=1}^r \HH_i)$ 
and use similarly formed notations for the induced maps in cohomology and homology.


\subsection{Compatibility of characteristic classes}\label{compatibility}


A natural bijection between the set of double coverings of a topological space $M$ and 
$H^1(M)$ consists in taking the first Stiefel-Whitney class of the associated real rank one  
vector bundle. Given a double covering $\pi: N\ra M,$ the corresponding cohomology class 
$b\in H^1(M)$ is called the {\it characteristic class of the covering}.

\smallskip

In our context, we denote by $b_0 \in H^1(\HH_0\setminus \partial \HH_0)=H^1(\HH_0)$ the 
characteristic class of the double covering $X\setminus X(\RR)\ra \HH_0\setminus\partial \HH_0,$ 
while  $b\in H^1(X(\RR)^{(2)}\setminus \Delta X(\RR))$ is the characteristic class of the double 
covering $X(\RR)\times X(\RR)\setminus \Delta X(\RR)\ra X(\RR)^{(2)}\setminus \Delta X(\RR).$ 
Notice that this double covering is trivial when restricted to the connected components 
$F_i\times F_j$  of $X(\RR)^{(2)}\setminus \Delta X(\RR)$  with $i\neq j.$ Accordingly, the 
restriction of $b$ to these components is trivial, and  $b$ can be viewed as an element of 
$H^1(\HH\setminus \partial \HH)=H^1(\HH).$ Under this identification, we see that 
$b=(b_1,\dots,b_r),$ where $b_i\in H^i(F_i^{(2)}\setminus \Delta F_i)=H^1(\HH_i)$ is the 
characteristic class of the the double covering 
$F_i\times F_i\setminus \Delta F_i\ra F_i^{(2)}\setminus \Delta F_i,$ for every $i=1,\dots,r.$

\smallskip

Let $\gamma\in H^1(E(\RR))$ denote the first Stiefel-Whitney class of the tautological 
line bundle of $E(\RR)=\PP_\RR(T^*X(\RR)).$ Due to the cut-and-paste construction 
and naturality of the characteristic classes, we have the following compatibility relations:
\begin{equation}
\label{compat-char-classes}
\inc_0^1(b_0)=\gamma=\inc^1(b).
\end{equation}


\subsection{The Betti numbers of $\HH_i$}


Here, we compute the Betti numbers of the strata $\HH_i$ of $X^{[2]}_{\rm main}(\RR)$ and 
the rank of the corresponding inclusion maps induced in homology/cohomology. To simplify 
the notation, throughout the rest of the paper  $\beta_k(X)=\beta_k(X(\CC))$ will always be 
denoted by $\beta_k,$ and, similarly,  we will use the notation $\beta_*$ for 
$\beta_*(X)=\beta_*(X(\CC)).$

\begin{lem}
\label{dimensions-H0}
If $X$ is a maximal $n$-dimensional real nonsingular projective variety, then 
the following formulas hold:
\begin{itemize}
\item[ 1)] $\displaystyle \beta_k(\HH_0)=\sum_{i=2n-k}^{2n} \beta_i,$ for every integer $0\leq k\leq n-1.$
\item[ 2)] $\displaystyle \beta_*(\HH_0)=\frac{n}{2}\beta_*.$
\end{itemize}
\end{lem}

\begin{proof} By Poincar\'e-Alexander-Lefschetz duality, we find
$$
\beta_r(\HH_0)=\beta_{2n-r}(X/\Conj,X(\RR)),
$$ 
and the proof of the claims follow
from applying Lemma \ref{RelativeQuotient} to the pair $(X,\Conj).$
\end{proof}

\begin{lem}
\label{inc-zero}
If $X$ is a maximal $n$-dimensional real nonsingular projective variety, then
$$
\rank (\inc_{0}^{k})=\sum_{i=0}^{k}\beta_i
$$
for every $k<n$.
\end{lem}

\begin{proof}
We argue by induction, and proceed by noticing first that $\rank (\inc_{0}^{0})=\beta_{0}.$ 
Indeed, from the short exact sequence 
$$
0\ra H^0(\HH_0,\partial \HH_0)\ra H^0(\HH_0)\ra\im (\inc_0^0)\ra 0,
$$
and Poincar\'e-Alexander-Lefschetz duality, we obtain 
\begin{align*}
\rank(\inc_0^0)=&\, \beta_0(\HH_0)-\beta_0(\HH_0,\partial \HH_0)\\
=&\, \beta_0(X/\Conj\setminus X(\RR))-\beta_0(X/\Conj, X(\RR))\\
=&\, \beta_{2n}(X/\Conj,X(\RR))-\beta_0(X/\Conj, X(\RR)).
\end{align*}
Since $X$ is maximal, from (\ref{relative-individual}) we obtain 
$\beta_{2n}(X/\Conj,X(\RR))=\beta_{2n}$ 
while $\beta_0(X/\Conj, X(\RR))=0.$
Therefore $\rank (\inc_{0}^{0})=\beta_{2n}=\beta_0.$

Consider next the cohomology long exact sequence of the pair 
$(\HH_0,\partial \HH_0).$ We find 
\begin{equation}
\label{coh-les}
0\ra \frac{H^{k-1}(\partial \HH_0)}{\im(\inc_{0}^{k-1})}\ra H^k(\HH_0,\partial \HH_0)
\ra H^k(\HH_0)\ra\im (\inc_0^k)\ra 0,
\end{equation}
for every $k\geq 1.$ By Poincar\'e-Alexander-Lefschetz duality, K\"unneth 
formula and (\ref{relative-individual}), from (\ref{coh-les}) we obtain:
\begin{align}
\label{master-inc}
 \rank(\inc_{0}^{k-1})+&\, \rank (\inc_{0}^{k})\notag\\
=&\, \beta_{k-1}(\partial \HH_0)+\beta_{k}(\HH_0)-\beta_k(\HH_0,\partial \HH_0) \notag\\
=&\, \beta_{k-1}(E(\RR))+\beta_{k}(X/\Conj\setminus X(\RR))
-\beta_k(X/\Conj,X(\RR))\notag\\
=&\, \sum_{i=0}^{k-1}\beta_i(X(\RR)) +\beta_{2n-k}(X/\Conj, X(\RR))
-\beta_k(X/\Conj,X(\RR))\notag\\
=&\, \sum_{i=0}^{k-1}\beta_i(X(\RR)) +\sum_{i=2n-k}^{2n}(\beta_i-\beta_i(X(\RR)))
-\sum_{i=k}^{2n}(\beta_i-\beta_i(X(\RR)))\notag\\
=&\, \sum_{i=0}^{k-1}\beta_i(X(\RR)) -\sum_{i=k}^{2n-k-1}\beta_i
+\sum_{i=k}^{2n-k-1}\beta_i(X(\RR))\notag \\
=&\, \sum_{i=0}^{2n-k-1}\beta_i(X(\RR)) -\sum_{i=k}^{2n-k-1}\beta_i.
\end{align}
Since $k<n$ and  $X$ is maximal, we find
$$
\sum_{i=0}^{2n-k-1}\beta_i(X(\RR))=\sum_{i=0}^{n}\beta_i(X(\RR))=\beta_*,
$$
and (\ref{master-inc}) can be rewritten as
\begin{equation}
\label{sum-ranks}
 \rank(\inc_{0}^{k-1})+ \rank (\inc_{0}^{k})= \sum_{i=0}^{k-1}\beta_i 
 +\sum_{i=2n-k}^{2n}\beta_i .
\end{equation}

Assume now  $\displaystyle\rank (\inc_{0}^{k-1})= \sum_{i=0}^{k-1}\beta_i.$ 
From (\ref{sum-ranks}) we find that 
$
\displaystyle
\rank (\inc_{0}^{k})=\sum_{i=2n-k}^{2n}\beta_i,
$
which, by Poincar\'e duality, is equivalent to 
$\displaystyle\rank (\inc_{0}^{k})= \sum_{i=0}^{k}\beta_i,$ 
concluding the induction argument.  
\end{proof}

Analogs of Lemmas \ref{dimensions-H0} and \ref{inc-zero} for $\HH_i,\,i=1,\dots,r,$  
are already available in the literature, in a wider context. For convenience, we collect 
below several results extracted  from Theorem 4.2 in \cite{square} and its proof.

\smallskip

Let $F$ be a compact $C^{\infty}$-manifold of real dimension $m.$ The
complement of the diagonal in its symmetric square $F^{(2)}\setminus \Delta F$ 
is naturally seen as the interior of a smooth compact
$2m$-dimensional manifold $\HH_F$ with boundary $\PP_\RR(T^* F).$ Let
$$
\inc^k_F: H^k(F^{(2)}\setminus \Delta F)=H^k(\HH_F)\ra H^k(\partial\HH_F)
=H^k(\PP_\RR(T^* F))
$$ 
be the restriction homomorphism.

\begin{thm}
\label{totaro-basis}
 Let $z_1, \dots, z_s$ be a basis for $H^*(F)$ and
let $Z_i$ be a closed pseudomanifold in $F$ that represents the class $z_i.$ Let 
$b\in H^1(F^{(2)}\setminus \Delta F)$ be the class of the double cover 
$g:(F\times F)\setminus \Delta F\ra F^{(2)}\setminus \Delta F.$ 
For every integer $k\geq 0,$ we have the following:
\begin{itemize}
\item[ 1)] A basis for $H^k(F^{(2)}\setminus \Delta F)$  is given by the elements 
$g_*(z_i\otimes z_j)$ with $\deg z_i +\deg z_j=k$ and $i < j,$ together with the elements 
$b^j[Z_i^{(2)}\setminus  \Delta Z_i]$ such that $2\deg z_i+j=k,\,i\geq 0$ and satisfying 
$0\leq j\leq m-1-\deg z_i.$
\item[ 2)] $\inc_F^k\left(g_*(z_i\otimes z_j)\right)=0$ for  all $i<j.$
\item[ 3)] The restrictions $\inc_F^k(b^j[Z_i^{(2)}\setminus  \Delta Z_i])$ satisfying 
$2\deg z_i+j=k,\,i\geq 0$  and $0\leq j\leq m-1-\deg z_i$ form a basis of   
$\im(\inc_F^k)\subseteq H^k\left(\PP_\RR(T^* F)\right).$
 \end{itemize}
\qed
\end{thm}

Adapted to our situation, we get the following consequences:

\begin{cor}
\label{beta-square}
If $X$ is an $n$-dimensional nonsingular real projective variety, then, for every $i=1,\dots,r$ 
and $k\geq 0,$ the following formulas hold:
\begin{itemize}
\item[ 1)] $\displaystyle \beta_{2k}(\HH_i)=\sum_{\substack{a+b=2k\\a<b}} 
\beta_a(F_i)\beta_b(F_i)+\frac12\beta_k(F_i)(\beta_k(F_i)-1)+\sum_{l=2k-n+1}^{k} \beta_l(F_i).$
\item[ 2)] $\displaystyle \beta_{2k+1}(\HH_i)=\sum_{\substack{a+b=2k+1\\a<b}} 
\beta_a(F_i)\beta_b(F_i)+\sum_{l=2k+1-n+1}^{k} \beta_l(F_i).$
\item[ 3)] $\displaystyle \beta_*(\HH_i)=\frac{1}2\beta_*(F_i)(\beta_*(F_i)-1)
+\sum_{k=0}^n (n-k)\beta_k(F_i).$
\item[ 4)] If, in addition, $X$ is maximal, then $\displaystyle \beta_*(\HH)=
\frac12\sum_{i=1}^r\beta_*^2(F_i)+\frac{n-1}{2}\beta_*.$ 
\end{itemize}
\end{cor}

\begin{cor}
\label{inc-sym}
If $X$ is an  $n$-dimensional real projective manifold, then
\begin{equation*}
\rank(\inc^m)=\sum_{[\frac{m}{2}]\geq k\geq m-n+1}\beta_k(X(\RR)),
\end{equation*}
for every $m\in\{0,\dots,  2n\}.$
\end{cor}

We conclude this section with the following results generalizing Proposition 3.2 in 
\cite{loss-surfaces}. The proof presented in  \cite[Proposition 3.2]{loss-surfaces} extends 
literally to higher dimensions and will be omitted.

\begin{prop}
\label{calculus}
Let $X$ be a compact complex manifold of dimension $n$ equipped with a real structure 
$\Conj.$ We have the following:
\begin{itemize}
\item[ 1)] The relation
\begin{equation}
\label{M-chi}
\chi(X^{[2]}(\RR))=\frac12\beta_*-\beta_{\rm odd}+\frac12\chi(X(\RR))^2-\chi(X(\RR)).
\end{equation}
\item[ 2)] If $\Tors_2 H_*(X;\ZZ)=0,$ 
the relation
\begin{equation}
\label{totaro-notorsion}
\beta_*(X^{[2]})=\frac12 \beta_*(\beta_*-1)+ n\beta_*-\beta_{\rm odd},
\end{equation}
\item[ 3)] If $\Tors_2 H_*(X;\ZZ)\neq 0,$ the relation 
\begin{equation}
\label{totaro-torsion}
\beta_*(X^{[2]}) \geq \frac12 \beta_*(\beta_*-1)+ n\beta_*-\beta_{\rm odd}.
\end{equation}
\end{itemize}
\end{prop}


\subsection{Special submanifolds of Hilbert squares}\label{special}


Let $X$ be a compact complex manifold of dimension $n$ equipped with a real structure 
$\Conj,$ and $Y\subseteq X$ a smooth, $\Conj$-invariant, complex submanifold  of 
codimension $m,$ with $Y(\RR)\neq \emptyset,$ and denote by $\upsilon\in H^m(X(\RR))$ 
the cohomology class of $Y(\RR).$ The Hilbert square $Y^{[2]}\subseteq X^{[2]}$ is a 
$\Conj$-invariant complex submanifold of codimension $2m.$ Its real locus,
$Y^{[2]}(\RR),$ is transversal to $E(\RR)$ and intersects the latter 
along $\PP_\RR(T^*Y(\RR)).$ The cycle
$Y^{[2]}(\RR)$  defines a cohomology class  $\Upsilon_Y$ in 
$H^{2m}({X}^{[2]}(\RR)).$

Let $j_0, \, j$ and $\kappa$ denote the inclusions $\displaystyle \HH_0
 \hookrightarrow X^{[2]}(\RR),\, \sqcup_{i=1}^r\HH_i\hookrightarrow X^{[2]}(\RR),$ 
and $E(\RR)\hookrightarrow X^{[2]}(\RR),$ respectively. For every integer $i\geq 0,$ 
consider the induced commutative diagram of restrictions:
\begin{equation}
\label{restrictions-diag}
\xymatrix{
 &H^{2i}(\HH_0)  \ar[dr]^{\inc^*_0} \\ 
  H^{2i}(X^{[2]}(\RR))\ar[rd]_{j^*} \ar[ur]^{j^*_0}\ar[rr]^{\hspace{5pt}\kappa^*}&& H^{2i}(E(\RR))\\
  & H^{2i}(\sqcup_{i=1}^r\HH_i)\ar[ur]_{\inc^*} }
\end{equation}

 Let 
\begin{align*}
\theta_{Y}=&j_0^*(\Upsilon_{Y})=[(Y/\Conj)\setminus Y(\RR)]\in  
H^{2m}(\HH_0)\\
\sigma_{Y}=&j^*(\Upsilon_{Y})=[Y(\RR)^{(2)}\setminus \Delta Y(\RR)]\in  
H^{2m}(\sqcup_{i=1}^r \HH_i).
\end{align*}
By the commutativity of diagram (\ref{restrictions-diag}), we have 
\begin{equation}
\label{equal-restrictions}
\inc_0^{2m}(\theta_Y)=\inc^{2m}(\sigma_Y)=\kappa^*(\Upsilon_Y).
\end{equation}

As is well-known, the cohomology ring $H^*(E(\RR))$ is 
the polynomial ring $H^*(X(\RR))[\gamma]/\langle \gamma^{n}+w_1\gamma^{n-1}
+\cdots +w_n \rangle,$ where $w_1,\dots, w_n$ are the Stiefel-Whitney classes 
of $X,$ and $\gamma$ is the first Stiefel-Whitney class of the tautological line 
bundle of $E(\RR)=\PP_\RR(T^*X(\RR))$ denoted by $\eta.$

Further on, to avoid excessive notations, we use  (where it does not lead to a 
misunderstanding) the same symbol for a cohomology class, or a vector bundle, 
and its restrictions.

\begin{lem} 
\label{restriction-hilb}
We have 
$$
\kappa^*(\Upsilon_{Y})=\gamma^{m}\upsilon+\gamma^{m-1}\Sq^1\upsilon+\dots 
+\gamma\Sq^{m-1}\upsilon+\Sq^{m}\upsilon
$$
where $\upsilon\in H^m(X(\RR))$ is the cohomology class of $Y(\RR).$
\end{lem}
\begin{proof}
The statement is the real analog of Lemma 6.1 in \cite{square}. The proof below 
follows the same lines.

Let $V={Y}^{[2]}(\RR)\cap E(\RR),$ and notice that 
$$
V=\PP_\RR(T^*Y(\RR))\subseteq \PP_\RR(T^*X(\RR))_{|Y(\RR)}\subseteq  
\PP_\RR(T^*X(\RR))=E(\RR)
$$
is the zero set of a transverse section of the vector bundle 
$
\Hom(\eta^{\vee},N)
$ 
over $\PP_\RR(T^*X(\RR))_{|Y(\RR)},$ where $N$  is the pullback to 
$\PP_\RR(T^*X(\RR))_{|Y(\RR)}$ of the normal bundle $N_{Y(\RR)/X(\RR)}$ of 
$Y(\RR)$ in $X(\RR).$ Indeed, such a section of $\Hom(\eta^{\vee},N)$ is provided 
by the subbundle $\eta$ of $T^*X(\RR)_{|Y(\RR)}.$ As a consequence, the cohomology 
class of $V$ in $\PP_\RR(T^*X(\RR))_{|Y(\RR)}$ is the top Stiefel-Whitney class of the 
rank $m$ vector bundle $\eta\otimes N:$
\begin{align}
\label{first-res}
w_m(\eta\otimes N)=&\,w_1^{m}(\eta)+w_1^{m-1}(\eta)w_1(N)+\cdots+w_{m}(N)\notag \\
=&\, \gamma^{m}+\gamma^{m-1}w_1(N)+\cdots+w_{m}(N).
\end{align}
Let now $s: Y(\RR)\ra X(\RR)$ denote the inclusion. Since by \cite[La formule g\'en\'erale]{thom}
 we have
$$
\Sq^i \upsilon=s_* w_i(N_{Y(\RR)/X(\RR)}),
$$
the conclusion of the lemma follows by pushing forward (\ref{first-res}) to $E(\RR)$.
\end{proof}


\section{Proof of Theorem \ref{converse}}


We follow here the same strategy as in the proof of \cite[Theorem 1.1]{loss-surfaces}. 
As a first step, we show the following:

\begin{prop}
\label{nonemptyness}
Let $X$ be a real nonsingular projective variety of dimension $n\ge 2$. If $X^{[2]}$ is maximal, 
then $X(\RR)\neq\emptyset.$
\end{prop}

\begin{proof}
By contradiction,  let us assume that $X^{[2]}$ is maximal and $X(\RR)=\emptyset$. 
Then $X^{[2]}(\RR)$ is the smooth quotient manifold $X/\Conj,$ of real dimension $2n,$ 
and so 
$$
\beta_*(X^{[2]}(\RR))=\beta_*(X/\Conj).
$$
Since $X(\RR)=\emptyset,$ the first relevant homology groups in the  Smith sequence in 
Theorem \ref{Smith.seq} satisfy
$$
0 \ra H_{2n}(X/\Conj)\xrightarrow{\Delta_{2n}}  H_{2n-1}(X/\Conj)\xrightarrow{\tr^*_{2n-1}}  
H_{2n-1}(X)\ra \cdots 
$$

Since $X/\Conj$ is connected, we have $\beta_{2n}(X/\Conj)=1$ and we find
$$
\beta_{2n-1}(X/\Conj) = 1+ \rank(\tr^*_{2n-1})\leq \beta_{2n}+\beta_{2n-1}
$$ 
where $\beta_i$ states for $\beta_i(X).$ We assume 
$$
\beta_{p+1}(X/\Conj)  \leq \sum_{k=p+1}^{2n}\beta_{k}.
$$ 
Using again the exactness of the Smith sequence, we find
$$
0\ra \im(\tr^*_{p+1})\   \ra H_{p+1}(X)\xrightarrow{\pr_{*,p}} H_{p+1}(X/\Conj)
\xrightarrow{\Delta_{p+1}}  
H_{p}(X/\Conj)\ra \im({\tr^*_{p}}) \ra 0
$$
Hence 
$$
\beta_p(X/\Conj)= \rank({\tr^*_{p}})+\beta_{p+1}(X/\Conj)+ \rank(\tr^*_{p+1})-\beta_{p+1}
\leq  \sum_{k=p}^{2n}\beta_{k}.
$$
By descending induction, we find that 
$$
\beta_{i}(X/\Conj)\leq\sum_{k=i}^{2n} \beta_{k}, \quad{\text{for all}}\,n\leq i\leq 2n.
$$
By Poincar\'e duality, or by inspecting the Smith sequence starting from the other end, 
we also find 
$$
\beta_{i}(X/\Conj)\leq\sum_{k=0}^{i} \beta_{k}, \quad{\text{for all}}\, ~0\leq i\leq n.
$$
Therefore, we obtain
$$
\beta_*(X/\Conj)\le  (n+1)\beta_0+n\beta_1+\dots +2\beta_{n-1}+\beta_n+2\beta_{n+1}
+\dots + (n+1)\beta_{2n}.
$$
Since $X$ is connected, we have
$\beta_0=\beta_{2n}=1$ and find
\begin{equation}
\label{estimateHilbR}
\beta_*(X^{[2]}(\RR))=\beta_*(X/\Conj)\le 2+n\beta_*.
\end{equation}
Also, since $X$ is projective, $\beta_2\geq 1$ and we notice that 
\begin{equation}
\label{bettis}
\beta_*\geq \beta_*-\beta_{\rm odd}\geq \beta_0+\beta_2+\beta_{2n}\ge3.
\end{equation}
To finish the proof, we use now  the third item of 
Proposition \ref{calculus}, 
(\ref{bettis}) and (\ref{estimateHilbR})
to notice that  
\begin{align*}
\beta_*(X^{[2]})\geq &\, \frac12\beta_*(\beta_*+1)+(n-1)\beta_* -\beta_{\rm odd}\\
\geq &\, 2 \beta_*+(n-1)\beta_* -\beta_{\rm odd}\\
= &\, n\beta_*+(\beta_*-\beta_{\rm odd})\\
> &\, 2+n\beta_*\\
\geq &\, \beta_*(X^{[2]}(\RR)),
\end{align*}
contradicting the maximality of $X^{[2]}.$ 
\end{proof}

\begin{rmk}
{\rm The projectivity  assumption can be replaced by ``compact complex" with 
$\beta_{2k}\ge 1$ for at least one $k, 0<k<n$".}
\end{rmk}

\begin{proof}[Proof of Theorem \ref{converse}] 
The non-emptiness of $X(\RR)$ being now ensured by Proposition \ref{nonemptyness}, 
the remaining part of the proof of Theorem \ref{converse} is identical 
to the proof of Theorem 1.1 in \cite{loss-surfaces}. We include the main ideas for 
convenience of the reader.

\smallskip

Pick a point $p\in X(\RR),$ whose existence is ensured by Proposition \ref{nonemptyness}, 
and consider the map $f:X\ra X^{(2)}$ given by 
$$
f(x)=\{p,x\}.
$$

Since the Hilbert-Chow map $\pi:X^{[2]}\ra X^{(2)}$ is an isomorphism when restricted to 
$X^{[2]}\setminus E$ and $f(X\setminus\{p\})\cap \pi(E)=\emptyset,$ the restriction of $f$ 
to $X\setminus\{p\}$ induces a map 
$$
\phi:X\setminus \{p\}\rightarrow X^{[2]}.
$$
The map $\phi$ extends to the blowup $q:\Bl_p X\ra X$ of $X$ at the point 
$p,$ and so we have a commutative diagram
$$
\xymatrix{\Bl_p X\ar[r]^\phi\ar[d]_{q}&X^{[2]\ar[d]^{\pi}}\,\\
 X
 \ar[r]^f &X^{(2)}.}
$$

By \cite[Lemma 4.2]{loss-surfaces} 
\footnote{Lemma 4.2 in \cite{loss-surfaces} is stated and proved for surfaces, 
but the same proof can be easily adapted in arbitrary dimension.}, 
the map 
$$
\phi^*:H^*(X^{[2]})\ra H^*(\Bl_p X)
$$ 
is surjective.
Consider next the commutative diagram
$$
\xymatrix{&H_G^*(X^{[2]})\ar[r]\ar[d]_{R^{[2]}} &H_G^*(\Bl_p X)\,\ar[d]^R\\
&H^*(X^{[2]})\ar[r]^{\phi^*\,\,\,}
&H^*(\Bl_p X),\\
}
$$
where the notation $H^*_G(Y)$ stands for the equivariant cohomology with 
$\FF_2$-coefficients of a topological space $Y$ equipped with the action of 
a group $G.$ In our case $G=\FF_2$ acting by complex conjugation on the 
corresponding complex variety. 

By  \cite[Proposition 2.6]{loss-surfaces}, since $X^{[2]}$ is maximal,  the restriction map 
$$
R^{[2]}: H_G^*(X^{[2]})\ra H^*(X^{[2]})
$$ 
is surjective, implying the surjectivity of the restriction map 
$$
R: H_G^*(\Bl_p X)\ra \ H^*(\Bl_p X).
$$ 
Applying once again \cite[Proposition 2.6]{loss-surfaces}, we find that $\Bl_p(X)$ is maximal, 
and so there remains to notice that $\Bl_p(X)$ is maximal if and only if $X$ is maximal. 
\end{proof}


\section{
Projective complete intersections}
\label{proof-odd}


This section is devoted to the proof of Theorem \ref{defect-odd}. We assume throughout this 
section that {\it $X$ is a maximal real nonsingular projective complete intersection of dimension $n$.}


\subsection{A reduction in the computation of the deficiency}


Our approach varies according to the parity of the dimension of $X.$


\subsubsection{Odd dimension.}


We assume first that the dimension $n$ of $X$ is odd. In this case, $\chi(X(\RR))=0$, 
since $X(\RR)$ is a closed odd dimensional manifold. Moreover, 
$\beta_{\rm odd}=\beta_n=\beta_*-(n+1),$ as it follows from the Lefschetz theorem 
on hyperplane sections. Therefore, using (\ref{M-chi}) and (\ref{totaro-notorsion}), 
the deficiency of $X^{[2]}$ can be computed as follows:
\begin{align}
\label{deficiency-square-odd}
\defi(X^{[2]})=&\, \beta_*(X^{[2]})-\beta_*(X^{[2]}(\RR))\notag\\ 
=&\, \beta_*(X^{[2]})-2\beta_{\rm even}(X^{[2]}(\RR)) +\chi(X^{[2]}(\RR))\notag\\ 
=&\, \frac12 \beta_*^2+ n\beta_*-2\beta_{\rm odd} - 2\beta_{\rm even}(X^{[2]}(\RR)) \notag\\
=&\, \frac12 \beta_*^2+ (n-2)\beta_*+2n+2-2\beta_{\rm even}(X^{[2]}(\RR)). 
\end{align}
Notice now that 
\begin{align}
\label{sum-betti-even}
\beta_{\rm even}(X^{[2]}(\RR))=&\, \beta_{\rm even}(X^{[2]}_{\rm main}(\RR))
+\beta_{\rm even}(X^{[2]}_{\rm extra}(\RR))\notag \\
=&\, 2 \sum_{l=0}^{\frac{n-1}2}\left(\beta_{2l}(X^{[2]}_{\rm main}(\RR))+ 
\beta_{2l}(X^{[2]}_{\rm extra}(\RR))\right).
\end{align}
To compute $\beta_{2l}(X^{[2]}_{\rm main}(\RR))$ for every integer $l$ such that $2l\leq n-1,$ 
we use the Mayer-Vietoris sequence
$$
\cdots  \ra H_{2l}(E(\RR))\xrightarrow{\mu_{2l}}\bigoplus_{i=0}^r H_{2l}(\HH_i)
\ra H_{2l}(X^{[2]}_{\text{main}}(\RR))\ra H_{2l-1}(E(\RR))\xrightarrow{\mu_{2l-1}} \dots,
$$
where $\mu$ is the inclusion map $\displaystyle \mu: E(\RR)\to \sqcup_{i=0}^r \HH_i$ 
(see Section \ref{cut-paste}). From the resulting short exact sequence
$$
0\ra\coker (\mu_{2l})\ra H_{2l}(X^{[2]}_{\text{main}}(\RR))\ra \Ker(\mu_{2l-1})\ra 0,
$$
we find
\begin{align}
\label{master-main}
\beta_{2l}(X^{[2]}_{\text{main}}(\RR))=&\, \Dim\coker (\mu_{2l})+\Dim\Ker(\mu_{2l-1})\notag \\
=&\, \beta_{2l}(\HH_0)+\sum_{i=1}^r\beta_{2l}(\HH_i)+\beta_{2l-1}(E(\RR))\notag \\
&\, -\rank(\mu_{2l})-\rank(\mu_{2l-1}).
\end{align}
Therefore, for every integer $l$ such that $0\leq 2l\leq n-1,$  we have
\begin{align}
\label{beta-2l-first}
\beta_{2l}(X^{[2]}(\RR))=\,& \beta_{2l}(X^{[2]}_{\rm main}(\RR))+
\beta_{2l}(X^{[2]}_{\rm extra}(\RR))\notag\\
=&\, \beta_{2l}(X^{[2]}_{\rm extra}(\RR))+\beta_{2l}(\HH_0)+
\sum_{i=1}^r\beta_{2l}(\HH_i)+\beta_{2l-1}(E(\RR))\notag \\
&\, -\rank(\mu_{2l})-\rank(\mu_{2l-1}).
\end{align}
Notice that by the K\"unneth formula we have 
\begin{align}
\label{extra-even}
\beta_{2l}(X^{[2]}_{\rm extra}(\RR))=&\, \sum_{1\leq s<t\leq r} \beta_{2l}(F_s\times F_t) \notag \\
=&\, \sum_{i+j=2l}\sum_{1\leq s<t\leq r}\beta_i(F_s)\beta_j(F_t).
\end{align}


\subsubsection{Even dimension.}


We assume here that the dimension $n$ of $X$ is even.
As it follows from the Lefschetz theorem on hyperplane sections, we have 
$\beta_{\rm odd}=0$. Therefore, using (\ref{M-chi}) and (\ref{totaro-notorsion}), the deficiency of 
$X^{[2]}$ can be computed as follows:
\begin{align*}
\defi(X^{[2]})=&\, \beta_*(X^{[2]})-\beta_*(X^{[2]}(\RR))\notag\\ 
=&\, \beta_*(X^{[2]})-2\beta_{\rm odd}(X^{[2]}(\RR)) - \chi(X^{[2]}(\RR))\notag\\ 
=&\, \frac12 \beta_*^2+ (n-1)\beta_* - 2\beta_{\rm odd}(X^{[2]}(\RR)) -\frac12 \chi(X(\RR))^2+\chi(X(\RR)).
\end{align*}
Since $X$ is maximal, we have $\chi(X(\RR))=\beta_*-2\beta_{\rm odd}(X(\RR)),$ and we find 
\begin{align}
\label{deficiency-square-even}
\defi(X^{[2]})=&\, n\beta_*+2\beta_*\beta_{\rm odd}(X(\RR))-2\beta_{\rm odd}^2(X(\RR))
-2\beta_{\rm odd}(X(\RR))\notag \\
&- \, 2\beta_{\rm odd}(X^{[2]}(\RR))\notag \\
=&\, n\beta_*+2\beta_{\rm even}(X(\RR))\beta_{\rm odd}(X(\RR))
-2\beta_{\rm odd}(X(\RR))\notag \\
&- \, 2\beta_{\rm odd}(X^{[2]}(\RR)).
\end{align}
By the same arguments as in the odd dimensional case, we get:
\begin{align}
\label{sum-betti-odd}
\beta_{\rm odd}(X^{[2]}(\RR))=&\, \beta_{\rm odd}(X^{[2]}_{\rm main}(\RR))
+\beta_{\rm odd}(X^{[2]}_{\rm extra}(\RR))\notag \\
=&\, 2 \sum_{l=1}^{\frac{n}2}\left(\beta_{2l-1}(X^{[2]}_{\rm main}(\RR))
+\beta_{2l-1}(X^{[2]}_{\rm extra}(\RR))\right)
\end{align}
where 
\begin{align}
\label{master-main-even}
\beta_{2l-1}(X^{[2]}_{\text{main}}(\RR))=&\, \Dim\coker (\mu_{2l-1})+\Dim\Ker(\mu_{2l-2})\notag \\
=&\, \beta_{2l-1}(\HH_0)+\sum_{i=1}^r\beta_{2l-1}(\HH_i)+\beta_{2l-2}(E(\RR))\notag \\
&\, -\rank(\mu_{2l-1})-\rank(\mu_{2l-2})
\end{align}
and 
\begin{align}
\label{extra-odd}
\beta_{2l-1}(X^{[2]}_{\rm extra}(\RR))=&\, \sum_{1\leq s<t\leq r} \beta_{2l-1}(F_s\times F_t) \notag \\
=&\, \sum_{i+j=2l-1}\sum_{1\leq s<t\leq r}\beta_i(F_s)\beta_j(F_t).
\end{align}

As a result we have now reduced the computation of the deficiency of the Hilbert 
square to the computation of the ranks of the restriction maps in the 
Mayer-Vietoris sequence.


\subsection{The rank of Mayer-Vietoris restriction maps}
\label{rank-restrictions}


Let $H$ be a smooth, $\Conj$-invariant, hyperplane section of $X,$ and denote by 
$H_{\ell},\, 1\leq \ell\leq n,$ the iterated hyperplane sections. Without loss of generality, we can 
and will assume that each $H_{\ell}$ is smooth and $\Conj$-invariant. Let $\upsilon\in H^1(X(\RR))$ 
denote the cohomology class of $H(\RR).$ For every $1\leq \ell\leq n,$ the Hilbert square 
$H_{\ell}^{[2]}$ is a $\Conj$-invariant, closed complex submanifold of complex codimension 
$2\ell$ in $X^{[2]}.$ The cycles $H_{\ell}^{[2]}(\RR),\, 1\leq \ell\leq n,$  define cohomology classes  
$\Upsilon_l$ in $H^{2\ell}(X^{[2]}(\RR)).$

Let $b_0 \in H^1(\HH_0\setminus \partial \HH_0)=H^1(\HH_0)$ be the 
characteristic class of the double covering $X\setminus X(\RR)\ra \HH_0\setminus\partial \HH_0,$ 
discussed in Section \ref{compatibility}. Using the notations from Section \ref{special}, we define 
$$
\theta_{\ell}= j_0^*(\Upsilon_l)\in H^{2\ell}(\HH_0),\, 1\leq \ell\leq n,
$$
and set $\theta_{0}=1\in H^{0}(\HH_0)\simeq \FF_2.$
\begin{lem}
\label{basis}
Let $X\subseteq \PP^N$ be a real, maximal, $n$-dimensional complete intersection. 
For every $k=0,\dots,n-1,$ the collection of classes $\left\{b_0^j\theta_{\ell}\right\}$ with
$j+2\ell=k$ form a basis of $H^k(\HH_0).$
\end{lem}

\begin{proof} By Poincar\'e-Alexander-Lefschetz duality and Lemma \ref{RelativeQuotient}, 
it suffices to show that for every $k=0,\dots, n-1,$ the elements in the set 
$\left\{b_0^j\theta_{\ell}\right\}$ with $j+2\ell=k$ are linearly independent.  
Let $c_\ell\in \FF_2, \, 0\leq\ell\leq [k/2]$ such that 
\begin{equation}
\label{ind-res}
\sum_{\ell=0}^{[k/2]} c_{\ell} b_0^{k-2\ell}\theta_{\ell}=0.
\end{equation}
Restricting (\ref{ind-res}) to $H^{k}(E(\RR))$ we find:
$$
0=\sum_{\ell=0}^{[k/2]} c_{\ell} \inc_0^*(b_0^{k-2\ell}\theta_{\ell}) =
\sum_{\ell=0}^{[k/2]} c_{\ell}\gamma^{k-2\ell} \inc_0^*j_0^*(\Upsilon_{\ell})=
\sum_{\ell=0}^{[k/2]} c_{\ell}\gamma^{k-2\ell} \kappa^*(\Upsilon_{\ell}).
$$
Using now Lemma \ref{restriction-hilb}, we obtain the following relation in $H^k(E(\RR)):$
$$
\sum_{\ell=0}^{[k/2]} c_{\ell}\gamma^{k-2\ell} (\gamma^{\ell}\upsilon^{\ell}+
\gamma^{\ell-1}\Sq^1\upsilon^{\ell}+\dots +\gamma\Sq^{\ell-1}\upsilon^{\ell}
+\Sq^{\ell}\upsilon^{\ell})=0.
$$
Notice that the left-hand side is a  degree $k$ polynomial $P(\gamma)$ 
in the variable $\gamma,$ with coefficients in $H^*(X(\RR)).$ Thus, according to the 
Leray-Hirsch theorem, the coefficients of each $\gamma^{k-r},\, 0\leq r\leq k$  must vanish, 
and we find: 
\begin{equation} 
\label{rec-indep}
\sum_{j=0}^{[r/2]}c_{r-j}\Sq^j\upsilon^{r-j}=0,\, {\text{for every}}\, r=0,\dots,k.
\end{equation}

By Lemma \ref{hyperplane-class}, $\upsilon^r\neq 0$  for every 
$r \leq \left[\frac{n}{2}\right],$ and the same holds for the pullbacks 
$\pi^*\upsilon^r\neq 0,$ as it follows, once more, from the Leray-Hirsch theorem.
Using this, from (\ref{rec-indep}) we inductively obtain $c_{[k/2]-\ell}=0$ for all 
$0\leq \ell\leq [k/2].$
\end{proof}

\begin{prop}
\label{inclusion-restrictions}
Let $X\subseteq \PP^N$ be a real, maximal, non-singular complete intersection of
dimension $n\geq 1.$ Then 
$$
\rank(\mu_{k})=\sum_{i=0}^{[k/2]}\beta_i(X(\RR)).
$$ 
for every $0\leq k\leq n-1.$
\end{prop}

\begin{proof} 
By (\ref{compat-char-classes}) and (\ref{equal-restrictions}), we find 
$\inc_0^*(b_0^j)=\inc^*(b^j)=\gamma^j$ and $\inc_0^* \left(\theta_{\ell}\right)\in \im(\inc^*)$ 
for every $j\geq 0$ and $0\leq \ell \leq \left[\frac{n}2\right],$ respectively. Therefore 
$\inc_0^* (b_0^j\theta_{\ell})\in \im(\inc^*)$ for every integer $j, \ell\geq 0$ such that $j+2\ell=k.$ 
In particular, we find $\im(\inc_0^{k})\subseteq\im(\inc^k),$ for every $0\leq k\leq n-1.$ 
The proof now follows from Corollary \ref{inc-sym}.
\end{proof}


\subsection{Proof of Theorem \ref{defect-odd}}


\subsubsection{Odd dimension}


For every integer $l$ such 
that $0\leq 2l\leq n-1,$  a direct computation using (\ref{extra-even}) 
and Corollary \ref{beta-square} shows 

\begin{align}
\label{single-quadratic-odd}
\beta_{2l}(X^{[2]}_{\rm extra}(\RR))+\sum_{i=1}^r\beta_{2l}(\HH_i)=&\, 
\sum_{i+j=2l}\sum_{1\leq s<t\leq r}\beta_i(F_s)\beta_j(F_t) 
 +\sum_{i=1}^r\sum_{\substack{a+b=2l\\a<b}} \beta_a(F_i)\beta_b(F_i)\notag \\
&+\, \sum_{i=1}^r\frac12\beta_l(F_i)(\beta_l(F_i)-1) +\sum_{i=1}^r\sum_{a=0}^{l} \beta_a(F_i)\notag\\
=&\, \frac12\sum_{a+b=2l}\beta_a(X(\RR))\beta_b(X(\RR))  \notag\\ 
&\, +\sum_{a=0}^{l-1}\beta_a(X(\RR))+\frac12\beta_l(X(\RR)).
\end{align}
According to Lemma \ref{dimensions-H0},
$$
\beta_{2l}(\HH_0)=\sum_{i=2n-2l}^{2n}\beta_i(X)=\sum_{j=0}^{2l}\beta_j(X)=l+1,
$$
so that, from  (\ref{single-quadratic-odd}) and Proposition \ref{inclusion-restrictions}, we 
infer that (\ref{beta-2l-first}) can be written as follows:
\begin{align}
\label{beta-2l}
\beta_{2l}(X^{[2]}(\RR))=&\, (l+1)+ \frac12\sum_{i+j=2l}\beta_i(X(\RR))\beta_j(X(\RR))
+ \sum_{a=0}^{l-1}\beta_a(X(\RR))  
\notag\\
&\,  +\frac12\beta_l(X(\RR))
+ \sum_{i=0}^{2l-1}\beta_i(X(\RR))-\sum_{i=0}^{l}\beta_i(X(\RR))-\sum_{i=0}^{l-1}\beta_i(X(\RR))\notag\\
=&\, (l+1)+ \frac12\sum_{i+j=2l}\beta_i(X(\RR))\beta_j(X(\RR)) \notag\\
&\,+ \sum_{i=0}^{2l-1}\beta_i(X(\RR))-\frac12\beta_l(X(\RR))-\sum_{i=0}^{l-1}\beta_i(X(\RR)).
\end{align}
From Lemma \ref{3together} and formula (\ref{beta-2l}) we find:

\begin{align*}
\beta_{\rm even}(X^{[2]}(\RR))=&\,2\sum_{l=0}^{\frac{n-1}{2}}\beta_{2l}(X^{[2]}(\RR))\notag\\
=&\, 2\sum_{l=0}^{\frac{n-1}{2}}(l+1)
+\sum_{l=0}^{\frac{n-1}{2}}\sum_{i+j=2l}\beta_i(X(\RR))\beta_j(X(\RR)) \notag\\
&\,+ 2\sum_{l=1}^{\frac{n-1}{2}}\sum_{i=0}^{2l-1}\beta_i(X(\RR))
-\sum_{l=0}^{\frac{n-1}{2}}\beta_l(X(\RR))
-2\sum_{l=1}^{\frac{n-1}{2}}\sum_{i=0}^{l-1}\beta_i(X(\RR))\notag \\
=&\,\frac{(n+1)(n+3)}{4}+\frac{ \beta^2_*}{4}
+\frac{(n-2)\beta_*}{2}-2\sum_{l=1}^{\frac{n-1}{2}}\sum_{i=0}^{l-1}\beta_i(X(\RR)).
\end{align*}
Therefore, the deficiency of $X^{[2]}$ is given by: 
\begin{align*}
\defi(X^{[2]})=&\,  \frac12 \beta_*^2+ (n-2)\beta_*+2n+2-2\beta_{\rm even}(X^{[2]}(\RR))\notag\\
=&\,  \frac12 \beta_*^2+ (n-2)\beta_*+2n+2 \notag\\
&\, -\frac{(n+1)(n+3)}{2}-\frac{ \beta^2_*}{2}-(n-2)\beta_*
+4\sum_{l=1}^{\frac{n-1}{2}}\sum_{i=0}^{l-1}\beta_i(X(\RR))\notag\\
=&\,  4\left(\sum_{l=1}^{\frac{n-1}{2}}\sum_{i=0}^{l-1}\beta_i(X(\RR))-\frac{n^2-1}{8}\right),
\end{align*}
which concludes the proof of Theorem \ref{defect-odd} when the dimension of $X$ is odd.


\subsubsection{Even dimension.}


For every integer $l$ such 
that $1\leq l\leq \frac{n}{2},$  a direct computation using (\ref{extra-odd}) 
and Corollary \ref{beta-square} shows 

\begin{align}
\label{single-quadratic-even}
&\beta_{2l-1}(X^{[2]}_{\rm extra}(\RR))+ \sum_{i=1}^r\beta_{2l-1}(\HH_i)=\notag \\
&\sum_{i+j=2l-1}\sum_{1\leq s<t\leq r}\beta_i(F_s)\beta_j(F_t) 
+\sum_{i=1}^r\sum_{\substack{a+b=2l-1\\a<b}} \beta_a(F_i)\beta_b(F_i)+
 \sum_{i=1}^r\sum_{c=0}^{l-1}\beta_c(F_i)=\notag \\
&\sum_{\substack{a+b=2l-1\\a<b}}\beta_a(X(\RR))\beta_b(X(\RR))  
+ \sum_{c=0}^{l-1}\beta_c(X(\RR)).
\end{align}
Notice from Lemma \ref{dimensions-H0} and the Lefschetz theorem on hyperplane sections that 
$
\displaystyle 
\beta_{2l-1}(\HH_0)=l.
$
By  (\ref{single-quadratic-even}) and Proposition \ref{inclusion-restrictions}, we now
infer that (\ref{master-main-even}) can be written as follows:
\begin{align}
\label{beta-2l-1}
\beta_{2l-1}(X^{[2]}(\RR))=&\, l
+ \sum_{\substack{a+b=2l-1\\a<b}}\beta_a(X(\RR))\beta_b(X(\RR))  \notag \\
&+\, \sum_{i=0}^{2l-2}\beta_i(X(\RR))-\sum_{i=0}^{l-1}\beta_i(X(\RR)).
\end{align}
From Lemma \ref{3together-even} and formula (\ref{beta-2l-1}) we find:

\begin{align}
\label{betti-odd-even-dim}
\beta_{\rm odd}(X^{[2]}(\RR))=&\,2\sum_{l=1}^{\frac{n}{2}}\beta_{2l-1}(X^{[2]}(\RR))\notag\\
=&\, 2\sum_{l=1}^{\frac{n}{2}}l 
+ \sum_{l=1}^{\frac{n}{2}}\sum_{\substack{a+b=2l-1\\a<b}}\beta_a(X(\RR))\beta_b(X(\RR)) \notag\\
&\,+ 2\sum_{l=1}^{\frac{n}{2}}\sum_{i=0}^{2l-2}\beta_i(X(\RR))
-2\sum_{l=1}^{\frac{n}{2}}\sum_{i=0}^{l-1}\beta_i(X(\RR))\notag \\
=&\,\frac{n(n+2)}{4}+\beta_{\rm even}(X(\RR))\beta_{\rm odd}(X(\RR)) 
+ \frac{n}{2}\beta_*-\beta_{\rm odd}(X(\RR)) \notag\\
&-2\sum_{l=1}^{\frac{n}{2}}\sum_{i=0}^{l-1}\beta_i(X(\RR)).
\end{align}
Using (\ref{deficiency-square-even})  and (\ref{betti-odd-even-dim}), 
we compute now the deficiency of $X^{[2]}:$  
\begin{align*}
\defi(X^{[2]})=&\, n\beta_*+2\beta_{\rm even}(X(\RR))\beta_{\rm odd}(X(\RR))
-2\beta_{\rm odd}(X(\RR))-2\beta_{\rm odd}(X^{[2]}(\RR))\notag\\
=&\,  n\beta_*+2\beta_{\rm even}(X(\RR))\beta_{\rm odd}(X(\RR))
-2\beta_{\rm odd}(X(\RR)) \notag\\
&\, -\frac{n(n+2)}{2}-2\beta_{\rm even}(X(\RR))\beta_{\rm odd}(X(\RR)) 
- n\beta_*+2\beta_{\rm odd}(X(\RR)) \notag\\
&+4\sum_{l=1}^{\frac{n}{2}}\sum_{i=0}^{l-1}\beta_i(X(\RR)\notag\\
=&\,  4\left(\sum_{l=1}^{\frac{n}{2}}\sum_{i=0}^{l-1}\beta_i(X(\RR))-\frac{n(n+2)}{8}\right),
\end{align*}
which concludes the proof of Theorem \ref{defect-odd} in the case when 
the dimension of $X$ is even.
\qed


\subsection{Proof of Theorem \ref{highdegree2kfolds}}


We assume first that $X^{[2]}$ is maximal. Then, by Theorem \ref{converse}, 
$X$ is maximal, and thus applying Corollary \ref{odd-coro} we conclude that 
$\beta_i(X(\RR))=1=\beta_{2i}$ for every $i$ such that $0\le i\le k-1$ with 
$k=\frac12 \dim X.$ From the maximality of $X$ and Poincar\'e duality we find  
that $\beta_k(X(\RR))=\beta_{2k}.$ This implies
\begin{align*}
\chi(X(\RR))&= \sum_{i=0}^{k-1}(-1)^i\beta_i(X(\RR))+(-1)^k\beta_{k}(X(\RR))
+\sum_{i=k+1}^{2k}(-1)^i\beta_i(X(\RR))\\
&=\sum_{i=0}^{k-1}(-1)^i+(-1)^k\beta_{2k}+\sum_{i=k+1}^{2k}(-1)^i.
\end{align*}
On the other hand,  according to Lefschetz trace formula we have
$$
\chi(X(\RR))=\sum_{i=0}^{k-1}(-1)^i+\tr H^{k,k}(X(\CC))+\sum_{i=k+1}^{2k}(-1)^i.
$$
Taking the difference, we obtain
$$
\beta_{2k}=(-1)^k\tr H^{k,k}(X(\CC))
$$
implying $\beta_{2k}=h^{k,k}(X(\CC)).$
However, the inequality $h^{k,k}(X(\CC))<\beta_{2k}$ holds for any nonsingular 
projective complete intersection of dimension $2k\ge 2$ except for linear subspaces,  
quadrics, intersections of two quadrics, and cubic surfaces in 
$\PP^3$ (see \cite{rapoport}).

\smallskip

Conversely, notice that linear subspaces $X\subset \PP^N$ satisfy $\beta_i(X(\RR))=1$ 
for all $i$ such that $1\leq i\leq \dim \,X.$ In the case when $X$ is a maximal real nonsingular 
quadric or a maximal real nonsingular intersection of two quadrics, then $\beta_i(X(\RR))=1$ 
for every integer $i$ such that $0\leq i\leq \left[\frac{n}{2}\right]-1$ ({\it cf.} \cite[section 7.1]{krasnov} 
or  \cite[Proposition 1.10]{tomasini} and \cite[Theorem 4.9]{tomasini}). Likewise, if $X$ is a 
maximal real nonsingular cubic surface in $\PP^3,$ then $X$ is the blow-up of $\PP^2$ at 
six real points, and so $\beta_0(X(\RR))=1.$ Thus, in all the four cases the conditions in 
Corollary \ref{odd-coro} are fulfilled. Hence, in all the four cases the Hilbert square is 
maximal.
\qed

\begin{rmk} {\rm The proof given above shows the maximality of the Hilbert square for real linear spaces 
and maximal real quadrics of any dimension, not only of even-dimensional ones.}
\end{rmk}


\section{Maximality criteria and the proof of Theorems \ref{connected} and \ref{Pn}}


We start by proving first the following maximality criterium. 

\begin{prop}
\label{mu-criterium}
Let $X$ be a maximal real nonsingular  $n$-dimensional  algebraic variety 
with $\Tors_2H_*(X,\ZZ)=0.$ Then the defect of $X^{[2]}$ is given by
$$
\defi(X^{[2]})=2\left(\rank \mu_*-\frac{n}2\beta_*-\frac12\beta_{\rm odd}\right)
\footnote{It is not difficult to check that in the case of surfaces this formula for 
the Smith-Thom deficiency of $X^{[2]}$ concords with
that given in \cite[Remark 5.4]{loss-surfaces}.}.
$$
In particular, $X^{[2]}$ is maximal if and only if 
$
\displaystyle
\rank \mu_*=\frac{n}2\beta_*+\frac12\beta_{\rm odd}.
$
\end{prop}

\begin{proof}
The Mayer-Vietoris sequence
$$
\cdots  \ra H_k(E(\RR))\xrightarrow{\mu_k}\bigoplus_{i=0}^r H_k(\HH_i)
\ra H_k(X^{[2]}_{\text{main}}(\RR))\ra H_{k-1}(E(\RR))\xrightarrow{\mu_{k-1}} \dots
$$
gives rise to a short exact sequence
$$
0\ra\coker (\mu_k)\ra H_k(X^{[2]}_{\text{main}}(\RR))\ra \Ker(\mu_{k-1})\ra 0,
$$
and so 
$$
\beta_k(X^{[2]}_{\text{main}}(\RR))= \Dim\coker (\mu_k)+\Dim\Ker(\mu_{k-1}).
$$
As a consequence, we have 
\begin{align*}
\beta_*(X^{[2]}_{\text{main}}(\RR))=&\, \sum_{k=0}^{2n} \beta_k(X^{[2]}_{\text{main}}(\RR))\\
=&\, \sum_{k=0}^{2n} \Dim \coker(\mu_k)+\Dim\Ker(\mu_{k-1})\\
=&\, \sum_{i=0}^r\beta_*(\HH_i)+\beta_*(E(\RR))- 2\rank (\mu_*).
\end{align*}
Furthermore, by Leray-Hirsch, we have 
$\beta_*(E(\RR))=\beta_*(\PP^{n-1}(\RR))\beta_*(X(\RR)).$ In our case, 
$X$ is maximal, and so 
\begin{equation}
\label{e=nb}
\beta_*(E(\RR))=n \beta_*.
\end{equation}
Using now Lemma \ref{dimensions-H0} and Corollary \ref{beta-square}, we find 
\begin{align}
\label{square-main}
\beta_*(X^{[2]}_{\text{main}}(\RR))=&\, \sum_{i=0}^r\beta_*(\HH_i)
+\beta_*(E(\RR))- 2\rank (\mu_*) \notag \\ 
=&\, \frac12\sum_{i=1}^r\beta_*^2(F_i)-\frac{1}{2}\beta_*
+2n\beta_*-2\rank( \mu_*). 
\end{align}
Notice now that by the K\"unneth formula we get
\begin{equation}
\label{square-extra}
\beta_{*}(X^{[2]}_{\text{extra}}(\RR))= \sum_{1\leq i<j\leq r}\beta_{*}(F_i\times F_j)=
\frac12(\beta_*^2-\sum_{1\le i\le r}\beta_*^2(F_i)).
\end{equation}
From (\ref{square-main}) and (\ref{square-extra}), we find 
\begin{align}
\label{square-plain}
\beta_*(X^{[2]}(\RR))=&\, \beta_{*}(X^{[2]}_{\text{main}}(\RR))
+\beta_{*}(X^{[2]}_{\text{extra}}(\RR)) \notag \\
=&\,  \frac12\sum_{i=1}^r\beta_*^2(F_i)-\frac{1}{2}\beta_*+2n\beta_*-2\rank (\mu_*)
+\frac12\beta_*^2-\frac12\sum_{1\le i\le r}\beta_*^2(F_i)\notag \\
=&\, \frac12\beta_*(\beta_*-1)+2n\beta_*-2\rank (\mu_*).
\end{align}
Invoking now the second item of Proposition \ref{calculus}, 
from (\ref{square-plain}) we find that the maximality defect of $X^{[2]}$ is
 \begin{align*}
 \defi(X^{[2]})=&\, \beta_*(X^{[2]})-\beta_*(X^{[2]}(\RR))\\
 =&\, 2\left(\rank(\mu_*)-\frac{n}2\beta_*-\frac12\beta_{\rm odd}\right).
\end{align*}
\end{proof}

Proposition \ref{mu-criterium} yields an effective criterion for the maximality 
of Hilbert squares in the absence of cohomology in odd degree. Before we 
state this criterion, we recall the following well known observation 
which goes back to Rokhlin and Thom.

\medskip

Let $M$ be a smooth compact $n$-dimensional manifold with boundary 
({\it neither $\partial M$ or $M$ is assumed connected}). Denote by 
$\mathfrak i: \partial M\to M$ the inclusion map, and by 
$\mathfrak i_k$ and ${\mathfrak i}_*$ the induced maps, 
$H_k(\partial M)\to H_k(M)$ and $H_*(\partial M)\to H_*(M),$ 
respectively.

\begin{lemma}
\label{boundary}
$\Dim \Ker (\mathfrak i_*) =\rank (\mathfrak i_*)=\frac12 \Dim \, H_*(\partial M)$.
\end{lemma}

\begin{proof} It suffices to show $\Dim \Ker (\mathfrak i_*) 
=\frac12 \Dim \, H_*(\partial M).$ The 
Poincar\'e-Lefschetz duality implies that for every $0\le k\le n-1$ the 
$k$-dimensional part of $\Ker (\mathfrak i_k)$ is the orthogonal complement of 
$\Ker (\mathfrak i_{n-1-k}).$
Herefrom, 
$$
\Dim \Ker (\mathfrak i_k)+\Dim \Ker (\mathfrak i_{n-1-k})= 
\beta_k(\partial M)=\beta_{n-1-k}(\partial M).
$$ 
The result follows by summing over $k$ from $0$ to $n-1$.
\end{proof}

\begin{prop}
\label{coh-criterium}
Let $X$ be a maximal $n$-dimensional real nonsingular projective variety with
$H_{\rm odd}(X)=0$. Then, $X^{[2]}$ is maximal if and only if one of the following 
two equivalent conditions holds
\begin{itemize}
\item[ 1)] $\displaystyle \Ker(\inc_{0*})=\Ker(\inc_*),$
\item[ 2)]  $\displaystyle \im (\inc_{0}^*)=\im (\inc_{}^*).$
\end{itemize}
\end{prop}
\begin{proof} By universal coefficient theorem, the assumption $H_{\rm odd}(X)=0$ 
implies $\Tors_2 H_*(X,\ZZ)=0$.
Thus, Proposition \ref{mu-criterium} applies and shows 
that $X^{[2]}$ is maximal if and only if $\displaystyle \rank(\mu_*)=\frac{n}{2}\beta_*.$  
As $\Dim \, H_*(E(\RR))=n\beta_*,$ this 
is equivalent to $\displaystyle \Dim\Ker(\mu_*)=\frac{n}{2}\beta_*.$ 
Notice now that $\Ker(\mu_*)=\Ker(\inc_{0*})\cap \Ker(\inc_{*})$
and, by Lemma \ref{boundary},
$\displaystyle \Dim \Ker(\inc_{0*})= \Dim \Ker(\inc_{*})=\frac{n}{2}\beta_*.$  
Therefore, the condition $\displaystyle \Dim\Ker(\mu_*)=\frac{n}{2}\beta_*$ is equivalent to 
$\Ker(\inc_{0*})=\Ker(\inc_{*}).$ 
By duality, this is equivalent to $\im (\inc_{0}^*)=\im (\inc_{}^*).$
\end{proof}

\begin{proof}[Proof of Theorem \ref{connected}]
Since $X^{[2]}$ is maximal, by Proposition \ref{coh-criterium}, we have 
$\im (\inc_{0}^*)=\im (\inc_{}^*).$ In particular, we find  $\im (\inc_{0}^k)=\im (\inc_{}^k),$ 
for every $k=\{0,\dots, 2n-1\}.$ If $k<n,$ using Corollary \ref{inc-sym} and Lemma 
\ref{inc-zero}, we find 
$$
\sum_{\ell=0}^{[k/2]} \beta_\ell(X(\RR))=\sum_{\ell=0}^k \beta_\ell.
$$
By induction, since $\beta_{2i+1}=0$ for all $i\geq 1,$ we find that
$$
\beta_k(X(\RR))= \beta_{2k} \quad{\text{for all}}\quad k<\frac{n}{2}.
$$ 
By Poincar\'e duality, this extends to $k>n/2.$ Finally, by Theorem \ref{converse}, 
$X$ is maximal, and so $\displaystyle\sum_{k=0}^n\beta_k(X(\RR))=\sum_{i=0}^n\beta_{2k},$ 
which implies $\beta_{k}(X(\RR))=\beta_{2k}$ if $n=2k.$
\end{proof}

\begin{proof}[Proof of Theorem \ref{Pn}] 
By Proposition \ref{coh-criterium} and Lemma \ref{boundary}, 
it suffices to show $\im (\inc_{0}^*)\supseteq\im (\inc_{}^*).$ 

\smallskip

Let $\{Y_\alpha\}_{\alpha\in I}$ be a finite collection of $\Conj$-invariant, smooth 
submanifolds of $X,$  such that the group $H^*(X(\RR))$ is generated by the Poincar\'e 
dual of the fundamental classes of $Y_\alpha(\RR), \alpha\in I.$ Without loss of generality, 
we can assume $Y_\alpha(\RR)\neq \emptyset$ for every $\alpha\in I.$ 
By Theorem \ref{totaro-basis}, $\im (\inc_{}^*)$ is generated by the elements  
$\inc^*(b^j[Y_\alpha^{(2)}(\RR)\setminus  \Delta Y_{\alpha}]),\,\alpha\in I.$ 
It remains to notice that (\ref{equal-restrictions}) in Section  \ref{rank-restrictions} 
implies that  for every $\alpha \in I$ and $j\geq 0$
$$
\inc^*(b^j[Y_\alpha(\RR)^{(2)}\setminus \Delta Y(\RR)]) 
= \gamma^j\kappa^*(\Upsilon_{Y_\alpha})= \inc_0^*(b_0^j\theta_{Y_\alpha}).
$$ 
\end{proof}


\bibliographystyle{alpha}

\end{document}